\documentclass[12pt,a4paper]{article}

\usepackage[english]{babel}
\usepackage[left=3cm,right=3cm,bottom=3cm,top=3cm]{geometry}

\usepackage[utf8]{inputenc}
\usepackage[english]{babel}
\usepackage{comment}
\usepackage{color}
\usepackage{amsfonts}
\usepackage{amsmath}
\usepackage{amssymb}
\usepackage{amsthm}
\usepackage{enumerate}
\usepackage{hyperref}

\newtheorem{theorem}{Theorem}
\newtheorem{proposition}{Proposition}
\newtheorem{lemma}[proposition]{Lemma}
\newtheorem{corollary}[proposition]{Corollary}
\theoremstyle{definition}
\newtheorem{definition}[proposition]{Definition}
\newtheorem{remark}[proposition]{Remark}

\newcommand{\R}{{\mathbb R}}
\newcommand{\N}{{\mathbb N}}
\newcommand{\Z}{{\mathbb Z}}
\newcommand{\Q}{{\mathbb Q}}

\newcommand{\expa}[1]{ \text{Fr}(#1)}
\newcommand{\+}[1]{{\mathcal {#1}}}

\newcommand{\card}{\#}
\newcommand{\measure}[1]{\lambda( #1 )}

\begin{document}
\title{Randomness and uniform distribution modulo one}
\author{Ver\'onica Becher  \and Serge Grigorieff }
\date{\today} 
\maketitle

\vspace*{-1cm}
{\footnotesize \tableofcontents}

\begin{abstract}
We elaborate the notions of Martin-L\"{o}f  and Schnorr randomness for real numbers
in terms of uniform distribution of sequences.  
We give a  necessary condition for a real number to be Schnorr random 
expressed in terms of classical uniform distribution of sequences. 
This extends  the result  proved by Avigad for sequences of 
 linear functions with integer coefficients
to the wider classical class of Koksma sequences of functions.
And, by requiring equidistribution with respect 
to every computably enumerable open set 
(respectively, computably enumerable open set with   computable measure)
in the unit interval, 
we give a sufficient condition for Martin-L\"of (respectively  Schnorr) randomness. 
\end{abstract}


\section{From randomness to uniform distribution}
How is the notion of randomness from algorithmic information theory 
related to the notion of uniform distribution from number theory?
In this paper we elaborate the definitions 
of Martin-L\"of randomness and Schnorr randomness for real numbers 
in  terms of uniform distribution of sequences.  

\subsection{Randomness}

Intuitively,  given a probability measure  such as Lebesgue measure, 
a real number  is random if  it has the properties of almost all real numbers; 
that is, if  it belongs to no set of  measure zero.
The intuition can not be taken literally
because every real number  belongs to the singleton set containing it,
which has  measure zero. To prevent the property being trivial
one can restrict to  computably defined sets,
as done by Martin-L\"of in~\cite{MartinLof1966}
and Schnorr in~\cite{Schnorr1971,DowneyGriffiths2004}.

A Martin-L\"of test for randomness 
is a sequence $(U_{n})_{n\in\N}$
of decreasing open sets $U_{n} = \bigcup_{k\in\N}I_{n,k}$
where $(I_{n,k})_{n,k\in\N}$ is a computable sequence 
of open intervals with rational endpoints
and such that the sequence of Lebesgue measures $(\measure{U_{n}})_{n\in\N}$ is 
upper-bounded by a computable function which converges to zero.
A real number passes this test if it is not in all $U_{n}$'s.
A real number is Martin-L\"of random if it passes every such test, i.e. if it is outside 
every one of these particular measure zero sets $\bigcap_{n\in\N}U_{n}$.
To get Schnorr randomness, consider 
those Martin-L\"of tests such that the sequence of reals $(\measure{U_{n}})_{n\in\N}$ is computable.
Martin-L\"of random reals constitute a proper subset of Schnorr random reals.
Since there are only countably many of these tests, 
there are also only countably many of these measure zero sets 
and this implies that almost all  real numbers, 
with respect to Lebesgue measure, are Martin-L\"of random 
and, a fortiori, Schnorr random.

The  definition entails that the fractional expansions 
of Martin-L\"of random  (or Schnorr random) real numbers 
obey all the usual probability laws.
And it follows that all the computable real numbers, 
such as the irrational algebraic numbers and  the usual mathematical constants 
including~$\pi$ and~$e$, are not Martin-L\"of random nor Schnorr random.
Chaitin~\cite{chaitin1975} showed that the halting probability of any
universal Turing machine with prefix-free domain 
is Martin-L\"of random.  
Many machine behaviours have Martin-L\"of random probabilities, 
for instance see~\cite{BDC2001,BG2005,BG2009,BFGM2006,Barmpalias2017,Ba17}.
For a  presentation of the theory of randomness see~\cite{DowneyHirschfeldt2010,Nies2009}.

\subsection{Uniform distribution modulo $1$}\label{ss:ud}
An infinite sequence $(x_n)_{n\geq 1}$ of real numbers is uniformly distributed
modulo $1$, abbreviated u.d. mod~$1$,  
if  the sequence formed by the fractional 
parts of its terms is equidistributed in the unit interval.
This means that for each subinterval of the unit interval, 
asymptotically, the proportion of terms falling within that subinterval 
is equal to its length.   
For a real number~$x$, we write $\lfloor x\rfloor$ to denote its integer part and 
$\{x\}=x- \lfloor x\rfloor$ for its fractional expansion.

Thus, a  sequence  $(x_n)_{n\geq 1}$ of real numbers is u.d. mod~$1$  if for 
every half open interval $[a,b)$ in $[0,1]$,

\begin{align}\label{eq xn ud}
\lim_{N\to\infty} 
\frac{1}{N} \#\big\{ n : 1\leq n \leq N  \text{ and } \{x_n\}\in [a,b)\big\} &= b-a.
\end{align}
A  presentation of  the theory of uniform distribution  can be read 
from~\cite{KuipersNiederreiter2006, DrmotaTichy1997,Bugeaud2012}.

In the years 1909 and 1910,  Bohl, Sierpi\'nski  and Weyl independently established  
that a  number $x$ is irrational 
exactly when  the sequence $(n\,x )_{n\geq 1}$ 
is u.d.~mod~$1$
(cf.~\cite[p.8 Example~2.1 and p.21 Notes]{KuipersNiederreiter2006}).
The subset of the irrational numbers
that Borel called  absolutely normal%
\footnote{A  real number is simply normal 
to an integer base $b$ greater than or equal to $2$ 
if in the fractional expansion of $x$ in  base $b$ 
every digit in $\{0,1,\ldots, b-1\}$ occurs with the same asymptotic frequency $1/b$.
A number is normal to base $b$ if it is simply normal to the bases $b, b^2, b^3, \ldots$
A number is absolutely normal if it is normal to all integer bases greater than or 
equal to~$2$.}~\cite{Borel1909} also has a characterization in terms of uniform distribution:
a real number is absolutely normal exactly when,
for every  integer $t$ greater than~$1$,
the sequence $(t^n x)_{n\geq 1}$ is u.d. mod~$1$,
see~\cite{KuipersNiederreiter2006}[p.70 Theorem 8.1 and p.74 Notes] and~\cite{Bugeaud2012}.

The Martin-L\"of random and Schnorr random real numbers are a proper subset of
 the  irrational numbers and also of the absolutely normal numbers,
we aim to find a  class of  sequences of functions $(u_n:[0,1]\to \mathbb R)_{n\geq 1}$
such that $x$ is Martin-L\"of random exactly when  $(u_{n}(x))_{n\geq 1}$ is u.d. mod 1.

The following  observation gives a first delimitation of what is possible
in terms of linear functions.

\begin{proposition}\label{p Avigad outside comput}
For every real $x$ there exists a strictly monotone increasing sequence 
$(a_n)_{n\geq1}$ of positive integers such that 
the sequence $(2^{a_n} x)_{n\geq1}$ is not u.d.
Moreover, if $x$ is computable then so is the sequence $(a_n)_{n\geq1}$.
\end{proposition}

\begin{proof}
If $x$ is dyadic, says $x=p/2^k$ with $p\in\N$ 
then the fractional part of $2^\ell x$ is $0$ for every $\ell\geq k$,
hence the left-hand side of equality~\eqref{eq xn ud}
is at most $k/N$ if $0<a<b$ and a fortiori smaller than $b-a$ for $N>k/(b-a)$. 
In particular, the sequence $(2^\ell x)_{\ell\geq k}$ is not u.d.

If $x$ is not dyadic, let the fractional part of $x$ be $\sum_{n\geq1} \delta_n/2^n$ 
where the $\delta_n$'s are the binary digits of $x$. Infinitely many 
of these digits are $0$ and infinitely many of them are $1$. 
Let $(a_k)_{k\geq1}$ be the sequence of ranks (in $\N\setminus\{0\}$)
of those which are $0$, i.e. $\delta_{a_k}=0$ for all $k\geq1$
and $\delta_n=1$ for all $n$ different from all the $a_k$'s.
Observe that the first digit of $2^{n-1}x$ is $\delta_n$. 
Thus, the fractional part of ${2^{a_k-1}x}$ is in $[0,1/2)$ for all $k\geq1$
hence the left-hand side of equality~\eqref{eq xn ud} is $1$ for all $N$ and does not tend to $1/2$.
This shows that the sequence $(2^{a_k-1}x)_{k\geq1}$ is not u.d.
Of course, if $x$ is computable so is the sequence  $(a_k)_{k\geq1}$.
\end{proof}

\subsection{Koksma's General Metric theorem}

We consider the class of sequences  in Koksma's General Metric theorem~\cite{Koksma1935}
but restricted to a suitable computability condition.
Koksma's General Metric Theorem, which can also be read  
from~\cite[Theorem~4.3 Chapter~1, p.34]{KuipersNiederreiter2006},
introduces the  class of so-called Koksma sequences of functions 
and establishes that if $(u_n)_{n\geq 1}$ is Koksma
then, for almost all real numbers $x\in [0,1]$, 
the sequence $(u_n(x))_{n\geq 1}$ is u.d. mod~$1$.

\begin{definition}[Koksma sequences of functions]\label{def:Koksma}
Let $K>0$.
A sequence $(u_n)_{n\geq 1}$ of functions $u_n: [0,1]\to\R$ is $K$-Koksma if
\begin{enumerate}
\item[-]
the function $u_n$ is continuously differentiable for every $n$,
\item[-] the difference $u'_m(x)-u'_n(x)$ is a monotone function on 
$[0,1]$ for all $m,n$,
\item[-] $|u'_m(x)-u'_n(x)|\geq K$ for all $m\neq n$ and all $x\in[0,1]$.
\end{enumerate}
A sequence $(u_n)_{n\geq 1}$ is Koksma if it is $K$-Koksma for some $K$.
\end{definition}
For example, if $t>1$ and $u_n(x)=t^n x$
then $(u_n)_{n\geq 0}$ is $(t-1)$-Koksma.
Similarly  if  $(a_n)_{n\geq 1}$ is a sequence of pairwise distinct integers 
and   $u_n(x)=a_n x$  then  $(u_n)_{n\geq 0}$  is $1$-Koksma.
In contrast, for $u_n(x)=x^n$, $(u_n)_{n\geq 0}$  is not Koksma
and $u_n(x)=x+ a$, for some $a$ is not Koksma either.
The importance of the Koksma class is stressed by the following classical result.

\begin{proposition}[Koksma's General Metric theorem, 1935]\label{p Koksma metric}
If the sequence of functions $(u_n)_{n\geq1}$ is Koksma then the set of reals $x\in[0,1]$ such that the sequence of reals $(u_n(x))_{n\geq1}$ is u.d. mod 1 has Lebesgue measure 1.
\end{proposition}

\subsection{Effective Koksma uniformly distributed reals}
The theory of computability formalizes the notion of function, or sequence of functions $\N\to\N$ or $[0,1]\to\R$ which can be effectively computed by some algorithm.
Using it, we can define  {\em effective Koksma sequences}
and {\em effective Koksma uniformly distributed reals}.

\begin{definition}[Effective Koksma functions]
A sequence $(u_n)_{n\geq 1}$ of functions is effective Koksma 
if it is Koksma and computable and if
$(u'_n)_{n\geq 1}$ is also computable.
\end{definition}

\begin{definition}[Effective Koksma uniformly distributed reals]
A real $x\in[0,1]$ is effective Koksma uniformly distributed (in short effective Koksma u.d.),
if, for every effective Koksma sequence $(u_n)_{n\geq 1}$, 
the sequence of reals $(u_n(x))_{n\geq 1}$ is u.d. mod 1.
\end{definition}

Proposition~\ref{p Koksma metric} immediately implies the following result.

\begin{proposition}\label{p Koksma ud ae}
The set of effective Koksma u.d. reals $x\in[0,1]$ has Lebesgue measure~1.
\end{proposition}

\subsection{Theorem~\ref{thm:1}:  Randomness implies effective Koksma u.d.}
Proposition~\ref{p Koksma ud ae} 
ensures that almost every Schnorr 
random  real is effective Koksma u.d.
In fact, almost everywhere can be replaced by everywhere. 

\begin{theorem}\label{thm:1}
Every Schnorr random real (a fortiori every Martin-L\"of random real) is effective Koksma u.d.
\end{theorem}

This theorem extends Avigad's~\cite[Theorem 2.1]{Avigad2013}, see also~\cite{fp2020} and the references there. Indeed, Avigad considers ``linear’’ Koksma sequences of functions, namely sequences of linear functions $(x\mapsto a_n x)_{n\geq1}$ where 
$(a_n)_{n\geq 1}$ is any computable sequence of pairwise distinct positive integers
(an assumption which trivially insures the effective Koksma condition). 
He proves that, for every Schnorr random real $x$, the sequence of 
reals $(a_n x)_{n\geq 1}$ is u.d. mod 1.

The proof of Theorem~\ref{thm:1} is done in Section~\ref{s:thm1}. 
It uses Koksma's General Metric Theorem. 

\begin{remark}
By Proposition~\ref{p Avigad outside comput} no computable real can be effective Koksma u.d.
Since there are Martin-L\"of random reals which are computable in $\emptyset'$, 
Theorem~\ref{thm:1} ensures that there 
exist effective Koksma u.d. reals which are computable in~$\emptyset'$.
\end{remark}

\subsection{Around Theorem~\ref{thm:1}}
We do  not know whether the converse fails or not,
 i.e. whether there exists an effective Koksma u.d. real which is not Schnorr random.
For linear functions with coefficients in $\N$, Avigad~\cite{Avigad2013} proves 
the following result.

First,  consider the classical notion of $\Sigma^0_1$ sets.

\begin{definition}\label{def:Sigma01}
A subset $S$ of $[0,1]$ is effectively open, or $\Sigma^0_1$, if it is of the form
$S=\bigcup_{k\geq 1} I_k$
for some computable sequence $(I_k)_{k\geq 1}$
of (possibly empty) open intervals $I_k$ with rational endpoints.
Complements of such sets are called effectively closed, or $\Pi^0_1$.
Schnorr $\Sigma^0_1$ sets are those $\Sigma^0_1$ sets 
with computable Lebesgue measure.
\end{definition}

\begin{proposition}[Avigad~\cite{Avigad2013}]\label{p Avigad ud not random}
There is an uncountable perfect $\Pi_{0}^{1}$ set $C\subset[0,1]$ of Lebesgue measure $0$
-- hence disjoint from the set of Schnorr random or simply Kurtz random reals -- 
and an atomless probability measure $\nu$ on $C$ such that,
for every sequence $(a_{n})_{n\geq1}$ of pairwise distinct positive integers,
the set $\{x\in C \colon \text{ $(a_{n}x)_{n\geq1}$ is u.d.}\}$ has $\nu$ measure 1.
In particular, there are reals which are not Schnorr random nor Kurtz random and such that $(a_{n}x)_{n\geq1}$ 
is u.d. for every computable sequence $(a_{n})_{n\geq1}$ of pairwise distinct positive integers.
\end{proposition}

Independently of the effective character of this result, 
it is related to  the existence of a measure with positive Fourier dimension,
cf.~\cite{Ekstrom-Schmeling}[Corollary 4, p.73].
We do not know whether the techniques around Koksma's General Metric Theorem 
and the Fourier dimension theory can be pushed to extend
 Proposition~\ref{p Avigad ud not random} by replacing sequences of linear functions $(x\mapsto a_{n}x)_{n\geq1}$ 
by general Koksma sequences of functions,
which would ensure that the converse of Theorem~\ref{thm:1}~fails.

\section{From $\Sigma^0_1$-uniform distribution to  randomness }
To capture Martin-L\"of and Schnorr randomness we strengthen the classical definition 
of uniform distribution.
With  this strengthening, 
 in Theorem~\ref{thm:Sigma01} we characterize
Martin-L\"of and Schnorr randomness in terms of uniform distribution modulo~1.

\subsection{$\+C$-uniform distribution}
A natural extension of uniform distribution is to replace intervals by a larger 
family of subsets of $[0,1]$.
Recall that $\lambda$ denotes Lebesgue measure. 

\begin{definition}[$\+ C$-u.d. mod~$1$]\label{def:S-ud}
Let $\+ C$ be a class of measurable subsets of $[0,1]$.
A~sequence $(x_n)_{n\geq 1}$ of real numbers is $\+ C$-uniformly distributed
modulo $1$, abbreviated $\+ C$-u.d.,  if
for every $A\in\+ C$,
\begin{align}\label{eq Cud}
\lim_{N\to \infty}\frac{1}{N}\card\Big\{n: 1\leq n\leq N, \{x_n\}\in A\Big\}=\lambda(A).
\end{align}
\end{definition}

The following two remarks delimitate some boundaries of the notion.

\begin{remark}[When $\+C$ is too small the notion adds nothing new]
In case  $\+C$ is the family of the rational intervals,
 or  it is the family of all Jordan-measurable subsets,
 or it is  the family of $\lambda$-continuous sets (which are measurable sets whose boundary has measure~$0$,
see~\cite[Chapter~1,p.5-6]{KuipersNiederreiter2006}),
then 
$\+C$-uniform distribution is exactly uniform distribution. 
\end{remark}

\begin{remark}[When $\+C$ is too large the notion is vacuous]\label{rk pitfall}
If the class $\+C$ contains all open subsets of $[0,1]$
then there is no sequence $(x_{n})_{n\geq1}$ that is $\+ C$-u.d.
Indeed, letting 
\begin{align*}
A&=[0,1) \cap \bigcup_{n\geq1} \left(\{x_n\} - 2^{-n-2}, \{x_n\} + 2^{-n-2} \right)
\end{align*}
we have $\lambda(A)\leq1/2$ whereas the left-hand side of equality~\eqref{eq Cud} 
is~$1$.
\end{remark}

On the other hand, the next proposition shows that the notion is not trivial if~$\+C$ is any countable family.
A sequence $(x_n)_{n\geq 1}$ of 
real numbers in the unit interval
can be viewed as a point in the Cartesian product $[0,1]^\N$. 
The Lebesgue measure $\lambda$ on $[0,1]$ induces the 
product measure $\lambda_\infty$ on $[0,1]^\N$.

\begin{proposition}\label{prop:ud ae}
Let  $\+ C$ be a countable class of measurable subsets of $[0,1]$.
Then  $\lambda_\infty$-almost all elements in $[0,1]^\N$ are  $\+C$-u.d.  
\end{proposition}

\begin{remark}
Proposition \ref{prop:ud ae} is an easy variant of a classical 
result due to Hlawka, 1956 (see  \cite[Theorem~2.2 Chapter~3 p.183]{KuipersNiederreiter2006}).
Let $\mu$ be a  probability measure on the Borel subsets
 of a compact space $X$ having a countable basis. 
Then $\mu_\infty(S)=1$ where $S$ is the set of sequences $(\{x_n\} )_{{n\in\N}}$ 
of elements of $X$ such that, for every real valued continuous function $f$ defined on $X$, we have
$\lim_{N\to\infty}{\frac {1}{N}}\sum _{n=1}^{N}f (\{x_n\} )=
  \int_{X} f$.
\end{remark}

\begin{proof}[Proof of Proposition~\ref{prop:ud ae}]
By Weyl's Criterion   (see ~\cite[Theorem 1.2]{Bugeaud2012}
or~\cite[Theorem 1.1 Chapter~1, p.2]{KuipersNiederreiter2006}),
a sequence $(x_n)_{n\geq 1}$ 
of real numbers is u.d.
 if, and only if, for 
every  real valued continuous function $f$ defined on $[0,1]$,
\begin{align}\label{eq mean f on the xn = integral f}
\lim_{N\to\infty}{\frac {1}{N}}\sum _{n=1}^{N}f (\{x_n\} )=
  \int_{0}^{1}f(x)\,dx.
\end{align}
We  use a result which follows from 
\cite[Lemma~2.1 Chapter~3 p.182]{KuipersNiederreiter2006},
where it is stated for bounded Borel functions
but its proof applies as well to bounded measurable functions $[0,1]\to\R$, 
hence it applies to our case:
if $f:[0,1]\to\R$ is a bounded measurable function
then, for $\lambda_\infty$-almost every sequence $(x_n)_{n\geq 1}$,
equality~\eqref{eq mean f on the xn = integral f} holds
(i.e. the mean value of $f$ on the $x_n$'s exists
and is equal to the integral of~$f$).

The case where $\+ C$ consists of just one measurable set $A$ is solved by the result above,
letting $f$ be the characteristic function of~$A$.
By countable additivity of $\lambda_\infty$, we get the case where $\+C$ is countable.
\end{proof}

\subsection{$\Sigma^0_1$-uniform distribution}

Introducing a computability constraint, 
the following definition avoids the pitfall of Remark~\ref{rk pitfall} and benefits from Proposition~\ref{prop:ud ae}.
Since the class $\Sigma^0_1$ (see Definition~\ref{def:Sigma01})
is a countable family of subsets in $[0,1]$, 
 Proposition~\ref{prop:ud ae} ensures that
almost all sequences $(x_n)_{n\geq 1}$ of real numbers 
are $\Sigma^0_1$-u.d.  

\begin{definition}
Let Schnorr $\Sigma^0_1$-u.d. mean $\+C$-u.d. where $\+C$ is the family of 
$\Sigma^0_1$ subsets of $[0,1]$ having computable measure.
 \end{definition}

Clearly, $\Sigma^0_1$-u.d. --- hence also Schnorr $\Sigma^0_1$-u.d. --- imply plain u.d.
Let's see that the notions are different.
\begin{proposition}
There is a sequence $(x_n)_{n\geq 1}$ that is u.d. 
but not Schnorr $\Sigma^0_1$-u.d.
\end{proposition}

\begin{proof}
Let $x$ be a  computable and irrational real number. 
As mentioned in \S\ref{ss:ud}, since $x$ is irrational 
the sequence~$(nx)_{n\geq 1}$ is u.d.;
we show that it is not Schnorr $\Sigma^0_1$-u.d.
Let $(b_{n,j})_{j\geq 1}$ be 
the expansion of $nx$ in base $2$, for $n=1,2,\ldots$.
Let $S$ be the following subset of~$[0,1]$,
\[
S= \bigcup_{n\geq 1} S_{n}
\text{\quad where } S_{n} = \big(0.b_{n,1} b_{n,2} \ldots b_{n, 2n}0,  \  0.b_{n,1} b_{n,2} \ldots b_{n, 2n}1\big)
\]
Since $x$ is computable, $S$ is an open computably enumerable subset of $[0,1]$.
Also, 
it  has computable measure:
the dyadic rational $\measure{\bigcup_{n<p} S_{n}}$ approximating $\lambda(S)$ up to
$ 2^{-2p+2}/3$ since 
\[
\measure{\bigcup_{n\geq p} S_{n}}\leq \sum_{n\geq p} 2^{-2n} = 2^{-2p+2}/3.
\]
By construction, for each $n\geq1$ the real $nx$ is in the $n$-th interval defining $S$, 
so that 
\[
\lim_{N\to \infty}\frac{1}{N} \#\Big\{n: 1\leq n\leq N, \{nx\}\in S\Big\} = 1.
\]
Since 
\[
\lambda(S) \leq \sum_{n\geq 1} 2^{-2n} = 1/3 < 1
\]
 Schnorr $\Sigma^0_1$-u.d. fails.
\end{proof}

\subsection{Theorem \ref{thm:Sigma01}: From $\Sigma^0_1$-u.d. to randomness}

We consider  the following classical notion.

\begin{definition}[Lipschitz function]\label{def ell Lipschitz}
For  $\ell>0$,
a function $f:[0,1]\to\R$ is \mbox{$\ell$-Lipschitz} if
$|f(x)-f(y)| \leq \ell |x-y|$ for every $x,y\in[0,1]$.
\end{definition}

\begin{theorem}[From $\Sigma^0_1$-u.d. to randomness]\label{thm:Sigma01}
Let $(u_n)_{n\geq 1}$ be a computable sequence of functions $[0,1]\to \R$
which is computably Lipschitz, i.e. $(\ell_n)_{n\geq 1}$-Lipschitz
for some computable sequence $(\ell_n)_{n\geq 1}$ of rationals.
If the real $x\in[0,1]$ is such that
$(u_n(x))_{n\geq 1}$ is $\Sigma^0_1$-u.d. (respectively Schnorr $\Sigma^0_1$-u.d.)
then $x$ is Martin-L\"of (respectively  Schnorr) random.
\end{theorem}

\begin{remark}\label{rk ell Lipschitz}
By the mean value theorem, every function $u:[0,1]\to\R$ continuously differentiable  
is $\ell$-Lipschitz for any $\ell \geq \max_{t\in[0,1]}|u'(t)|$.
In particular, the class of computable sequences of computably Lipschitz functions $[0,1]\to \R$ is far wider than the Koksma class.
\end{remark}
%

The proof of Theorem~\ref{thm:Sigma01} is done in Section~\ref{s:Sigma01}.

\section{Characterization of randomness  with $\Sigma^0_{1}$-u.d.}

\subsection{From randomness to $\Sigma^0_1$-u.d.}

Let's consider  these classical notions and  results on measure-preserving ergodic operators.

\begin{definition}
\begin{enumerate}
\item Let $T\colon [0,1)\to[0,1)$. A set $A\subseteq[0,1)$ is almost invariant if
$A$ and $T^{-1}(A)$ coincide up to a Lebesgue measure $0$ set.
\item A measure-preserving operator $T\colon [0,1)\to[0,1)$ is  ergodic
 if every almost invariant set has Lebesgue measure~$0$ or~$1$. 
\end{enumerate}
\end{definition}

\begin{remark}
Two simple examples of ergodic maps on $[0,1)$ are 
the shift $x\mapsto2x\mod 1$ (corresponding to the shift on the Cantor space $2^\N$),
and the translation $x\mapsto x+a\mod1$ when~$a$ is irrational 
(cf.~\cite{Shiryaev}, Chap. 5, Example p.408). Observe that, in this case, the continuity implied by computability is relative to the topology of the disk: 
a typical neighborhood of $0$ is a set $[0,\theta)\cup(1-\theta,1)$.
\end{remark}

\begin{proposition}[Birkhoff and Khinchin Theorem, 1931]
\label{thm:Birkhoff and Khinchin} 
If $T\colon [0,1)\to[0,1)$ is measure-preserving and ergodic 
and $A\subseteq[0,1)$ 
is Lebesgue measurable then 
for almost all $x\in [0,1)$,
\begin{align*}
\lim_{N\to \infty}\frac{1}{N}  \#\Big\{n: 1\leq n\leq N, T^{n}(x)\in A\Big\}=\lambda(A).
\end{align*}
\end{proposition}

A decade ago this theorem has been effectivized, first
by P. G\'acs et al.~\cite{GHR11}, 2011 for Schnorr randomness, 
and then 
by L. Bienvenu et  al.~\cite[Theorem 8]{BDHMS2012} 
and by J. Franklin et  al.~\cite[Theorem 6]{FGMN2012}, 2012 
for Martin-L\"of randomness.

\begin{proposition}[{E}ffective Birkhoff Ergodic Theorem\cite{GHR11,BDHMS2012,FGMN2012}]
\label{thm:ergodic} 
Let $T$ be a computable ergodic operator $[0,1)\to[0,1)$
and let $U \subseteq [0,1)$ be a $\Sigma^0_{1}$ (respectively 
 $\Sigma^0_{1}$ with computable measure)  set.
For every Martin-L\"of (respectively  Schnorr) random
real  $x\in[0,1)$ we have
\begin{align*}
\lim_{{N\to+\infty}} 
\dfrac{1}{N} \#\{n \colon 1\leq n\leq N,\ T^{n}(x)\in U\}
&=\lambda(U).
\end{align*}
In other words, if $x$ is Martin-L\"of (respectively  Schnorr) random 
then the sequence $(T^{n}(x))_{n\geq1}$ is $\Sigma^0_{1}$-u.d. 
(respectively Schnorr $\Sigma^0_{1}$-u.d.)
\end{proposition}

\begin{remark}
In \cite[Theorem 6]{FGMN2012} the result is stated for $\Pi^0_{1}$ sets and it is a full characterization.
The passage to complements, 
 i.e. to $\Sigma^0_{1}$ sets,
 is obvious. 
\end{remark}

%
\subsection{\mbox{Theorem~\ref{thm:conclusion}: Characterization of randomness with $\Sigma^0_{1}$-u.d.}}
%
Grouping Theorem~\ref{thm:Sigma01} and the Effective  Birkhoff Ergodic Theorem
due to~\cite{GHR11,BDHMS2012,FGMN2012} and stated in 
Proposition~\ref{thm:ergodic}, we obtain our third result, Theorem~\ref{thm:conclusion} .

\begin{theorem}\label{thm:conclusion}
Let  $x\in[0,1)$.
The following conditions are equivalent.
\begin{enumerate}
\item[i.]
There exists a computable sequence $(u_{n})_{n\geq1}$ of functions $[0,1]\to\R$ 
which is computably Lipschitz, 
i.e. $(\ell_{n})_{n\geq1}$-Lipschitz with $(\ell_{n})_{n\geq1}$ 
a computable sequence of rationals,
such that the sequence $(u_{n}(x))_{n\geq1}$ is $\Sigma^0_{1}$-u.d.
\item[ii.] The sequence  $(x+na)_{n\geq1}$, for some irrational number $a$,  is $\Sigma^0_{1}$-u.d.
\item[iii.]
The sequence  $(2^{n}x)_{n\geq1}$ is $\Sigma^0_{1}$-u.d.
\item[iv.]
For every computable measure-preserving and ergodic operator 
$T\colon[0,1)\to[0,1)$, 
the sequence  $(T^{n}(x))_{n\geq1}$ is $\Sigma^0_{1}$-u.d.
\item[v.]
$x$ is Martin-L\"of random.
\end{enumerate}
The same equivalences are valid when replacing $\Sigma^0_{1}$-u.d. by
Schnorr $\Sigma^0_{1}$-u.d. and Martin-L\"of randomness by Schnorr randomness.
\end{theorem}

\begin{proof}
$(i)\Rightarrow(v)$ is our Theorem~\ref{thm:Sigma01} and
$(v)\Rightarrow(iv)$ is the effective version of Birkhoff Ergodic Theorem,
see~Proposition~\ref{thm:ergodic}.

For $(iv)\Rightarrow(iii)$ consider the operator $T\colon[0,1)\to[0,1)$
 such that $T(x)=2x\mod 1$. It is computable measure-preserving and ergodic. 
Observe that $T^{n}(x) = 2^n x \mod 1$ 
hence the two sequences $(T^{n}(x))_{n\geq1}$ and $(2^n x)_{n\geq1}$ are simultaneously uniformly distributed modulo $1$
or not relative to any fixed subset.

Implication $(iv)\Rightarrow(ii)$ is similar.

Implications $(iii)\Rightarrow(i)$ and $(ii)\Rightarrow(i)$ are straightforward.
\end{proof}

\begin{remark}
Observe that $(x\mapsto 2^n x)_{n\geq1}$ 
is a Koksma sequence of functions $[0,1]\to\R$
whereas $(x\mapsto x+na)_{n\geq1}$ is not Koksma.
\end{remark}

\begin{remark}
As aforementioned  in Section~\ref{ss:ud}, the sequence  $(na)_{n\geq1}$ is u.d. when $a$ is irrational hence,
for every real $x$, so is $(x+na)_{n\geq1}$ .
Also, $(2^{n}x)_{n\geq1}$ is u.d. exactly when $x$ is normal to base $2$.
But being $\Sigma^0_{1}$-u.d. or 
Schnorr $\Sigma^0_{1}$-u.d is a stronger condition.
\end{remark}

\begin{remark}
We cannot replace the existential quantification in condition (i) of 
Theorem~\ref{thm:conclusion} by a universal one: 
consider a sequence $(u_n)_{n\in \N}$ such that for all~$n$, $u_n$ is the same
Lipschitz function, then $(u_{n}(x))_{n\geq1}$ is constant, hence is not u.d.
and a fortiori not $\Sigma^0_{1}$-u.d.
\end{remark}

\section{Proof of Theorem~\ref{thm:1}}\label{s:thm1}
%
\subsection{Variations around Koksma's General Metric Theorem}
%

\begin{lemma}
\label{l Weyl variant}\mbox{ }
We follow Vinogradov's notation $e(x)$ for~$e^{2i\pi x}$.
For $\mathbf{x} = (x_n)_{n\geq 1}$ a sequence of reals, we let
$S_N(\mathbf{x})  = \dfrac{1}{N}\sum_{j=1}^{N}e(x_j)$
and we denote $h\mathbf{x}$ the sequence of reals $(hx_n)_{n\in\N}$.

The following  are equivalent:
\begin{enumerate}
\item The sequence $(x_n)_{n\geq 1}$ is u.d. mod~$1$.
\item For every $h\in\Z\setminus\{0\}$, $\lim_{n\to\infty}S_N(h\mathbf{x})  = 0$.
\item There exists a strictly increasing sequence of positive integers $(M_k)_{k\in\N}$ such that
$M_{k+1}-M_k = o(M_k)$ and 
$\lim_{k\to\infty}S_{M_k}(h\mathbf{x})  = 0$ for every $h\in\Z\setminus\{0\}$.
\end{enumerate}
\end{lemma}

\begin{proof}
The equivalence of (1) and (2) is Weyl's Criterion
(\cite[Theorem 2.1 Chapter~1, p.2]{KuipersNiederreiter2006}). 
Letting $M_k=k$, implication $(2)\Rightarrow(3)$ is trivial. 

To prove its converse, observe that $|e(x)|=1$ for all $x$,
so that the modulus of a sum of $p$ such exponentials is bounded by $p$.
Thus, if $M_k\leq N<M_{k+1}$, 
\begin{align*}
| S_N(h\mathbf{x}) - S_{M_k}(h\mathbf{x}) |
&= \left| \left(\dfrac{1}{N} - \dfrac{1}{M_k}\right) \sum_{\ell=1}^{\ell=M_k} e(hx_\ell)
+ \dfrac{1}{N} \sum_{\ell=M_k+1}^{\ell=N} e(hx_\ell) \right|
\\&\leq \dfrac{2(N-M_k)}{N}
\\
&\leq  \dfrac{2(M_{k+1}-M_k)}{M_{k+1}},
\end{align*}
the last inequality holds because  $\dfrac{t-M_k}{t}$ is increasing in $t\geq M_{k}$.
Thus, 
\[
|S_N(h\mathbf{x})| \leq |S_{M_k}(h\mathbf{x}) | + \dfrac{2(M_{k+1}-M_k)}{M_{k+1}}.
\]
The hypothesis ensures that  in the last expression both terms tend to $0$. 
\end{proof}

\begin{remark}\label{rk oMk oMk+1}
Clearly, $M_{k+1}-M_k = o(M_k)$ implies $M_{k+1}-M_k = o(M_{k+1})$.
The converse is also true: arguing by contraposition, 
if $M_{k+1}-M_k \geq\varepsilon M_k$ for infinitely many $k$'s then also
$M_{k+1}-M_k \geq \dfrac{\varepsilon}{1+\varepsilon}  M_{k+1}$ for the same $k$'s.
\end{remark}

\begin{corollary}\label{cor Weyl variant}
Let $x$ be a real and 
let $(M_k)_{k\in\N}$ be a strictly increasing sequence of positive integers 
such that $M_{k+1}-M_k = o(M_k)$. 
The following are equivalent:
\begin{enumerate}
\item
For every effective Koksma sequence $\mathbf{u}=(u_n)_{n\geq1}$, 
 $(u_n(x))_{n\geq 1}$ is u.d.
\item
For every effective Koksma sequence $\mathbf{u}$, 
$\lim_{k\to\infty}S_{M_k}(\mathbf{u}(x))  = 0$.
\end{enumerate}
\end{corollary}

\begin{proof}
Conditions 1 and 2 come from 1 and 3  in Lemma~\ref{l Weyl variant} 
where $h$ is removed. 
This is justified because  if $\mathbf{u}$ is an effective Koksma sequence
so is $h\mathbf{u}$ for~$h\neq0$.
\end{proof}

\begin{lemma}[Around Koksma's General Metric Theorem]\label{l bound int SN2}
For every integer $N\geq3$, for every $K>0$, 
for every integer $h\neq0$,
for every $K$-Koksma sequence $\mathbf{u} = (u_n)_{n\geq1}$,
denoting $h\mathbf{u}(x)$ the sequence of reals $(hu_n(x))_{n\in\N}$,
\begin{align}\label{bound int SN2}
\int_0^1 |S_N(h\mathbf{u}(x))|^2 dx 
&<  \dfrac{1}{N}  + \dfrac{8}{|h|K} \dfrac{\ln(3N)}{N}
< \left(1+\dfrac{17}{|h|K}\right) \dfrac{\ln(N)}{N}
\end{align}
\end{lemma}

\begin{proof}
The last inequality of the proof of Koksma's General Metric Theorem 
in~\cite[Theorem 4.3, Chapter~1, p.34-35]{KuipersNiederreiter2006} is the first above inequality. 
Then observe that $8\ln(3)<9$.
\end{proof}

Inequality~\eqref{bound int SN2} immediately implies
the inequality in the corollary below.

\begin{corollary}\label{cor bound mu SN}
Let $\mathbf{u}$ be a $K$-Koksma sequence.
Then, for $\varepsilon>0$,
\begin{align*}
\lambda(\{x\in[0,1] \mid |S_{N}(h\mathbf{u}(x))| \geq \varepsilon\}) 
< \dfrac{1}{\varepsilon^2} \left(1+\dfrac{17}{|h|K}\right) \dfrac{\ln(N)}{N}.
\end{align*}
\end{corollary}

%
\subsection{Solovay tests for randomness}

We consider  the notions of Solovay randomness and total Solovay 
randomness, see~\cite{Solovay1975,DowneyHirschfeldt2010}.
A Solovay test is a computable sequence
$(V_n)_{n\geq 1}$ of computably enumerable
open subsets of  real numbers such that
$\sum_{n\geq 1}\lambda(V_n)$ is finite.
A real~$x$ passes the test if it belongs just to finitely many $V_n$'s.
A real~$x$ is Solovay random if it passes every Solovay test.
Requiring that $\sum_{n\geq1}\lambda(V_n)$ be a computable real,
we get the notions of total Solovay test and total Solovay randomness.

\begin{proposition}[\protect{\cite[Proposition 3.2.19]{Nies2009} and
\cite[Theorem 7.1.10]{DowneyHirschfeldt2010}}]
\label{p Solovay random}
A real passes all Solovay tests if and only if it passes all Martin-L\"of tests.
A real passes all total Solovay tests if and only if it passes all Schnorr tests.
Thus, Solovay and Martin-L\"of randomness coincide 
and total Solovay and Schnorr randomness also coincide.
\end{proposition}

%
\subsection{Proof of Theorem~\ref{thm:1} for Martin-L\"of randomness}
%
For clarity of exposition, we first prove the weak form of Theorem~1 which states that every Martin-L\"of random real is effective Koksma u.d.
We argue by contradiction.
Let $(u_n)_{n\geq 1 }$ be an effective Koksma sequence and assume the real $x\in[0,1]$ is such that 
the sequence $(u_n(x))_{n\geq 1}$ is not u.d. mod~$1$.
We show that $x$ is not Schnorr random by constructing a total Solovay test failed by $x$ (cf. Proposition~\ref{p Solovay random}).

The argument extends in a simple way Avigad's argument in~\cite{Avigad2013}
using Corollary~\ref{cor bound mu SN}.
For $n$ and $h$, let 
\[
S_{N,h}(\mathbf{u}(t))=\frac{1}{N}\sum_{j=1}^{N}e(h\, u_j(t)).
\]
Letting $M_k=k^2$, we have $M_{k+1}-M_k=2k+1=o(M_k)$.
Applying Lemma~\ref{l Weyl variant}, since we assumed that 
$(u_n(x))_{n\geq 1}$ is not u.d. mod~$1$,
there exists an integer $h\neq0$,
such that $S_{k^2,h}(x)$ does not tend to $0$ when $k$ goes to infinity.
Thus, there exists a strictly positive rational $\varepsilon$ 
and infinitely many $k$'s such that $|S_{k^2,h}(x) |>\varepsilon$.
Let 
\begin{align}\label{eq Akvarepsilon}
A_k^{\varepsilon} = \{t\in[0,1]: |S_{k^2,h}(\mathbf{u}(t))|>\varepsilon\}.
\end{align}
Thus, $x$ belongs to infinitely many $A_k^{\varepsilon}$'s.

Applying Corollary~\ref{cor bound mu SN} with $N=k^{2}$, we get a factor $2$ from $\ln(k^{2})$
hence,
\begin{align}\label{eq mu Akvarepsilon}
\measure{A_k^{\varepsilon}} 
< \dfrac{2}{\varepsilon^2}  \left(1+\dfrac{17}{|h|K}\right) \dfrac{\ln(k)}{k^{2}}
\end{align}
As a consequence, since the series $\sum_{k\geq1}\ln(k)/k^2$ converges
so does the series 
\linebreak
$\sum_{k\geq1}\measure{A_k^{\varepsilon}}$.
It remains to check that  the sets $A_k^{\varepsilon}$ are $\Sigma^0_1$ uniformly in $k$. 
This is immediate from the fact that the sequence $(u_j)_{j\in\N}$ is computable.
Hence, the sequence $(A_k^{\varepsilon})_{k\geq1}$ is a Solovay test.
Now, since $x$ belongs to infinitely many $A_k$'s, it fails this Solovay test.
We conclude that  $x$ is not Martin-L\"of random.

Observe that the left cut of the real $\sum_{k\geq1}\measure{A_k^{\varepsilon}}$ (the set of rational numbers less than this sum) is computably enumerable but not necessarily computable
so that $(A_k^{\varepsilon})_{k\geq1}$ may not be a total Solovay test.
\hfill $\qed$
%
\subsection{Proof of Theorem~\ref{thm:1} for Schnorr randomness}
%
To contradict Schnorr randomness requires some addition to the previous argument.
We shall modify the open sets $A_k^{\varepsilon}$ to open sets $B_k^{\varepsilon}$
which are finite unions of rational intervals.
This will allow to get the computability 
of the sum $\sum_{n\geq1} \lambda(B_k^{\varepsilon})$, 
ensuring that $(B_k^{\varepsilon})_{k\geq1}$ is a total
Solovay test failed by $x$ hence that $x$ is not Schnorr random.

We shall consider both sets $A_k^\varepsilon$ and $A_k^{\varepsilon/2}$, 
cf. formula~\eqref{eq Akvarepsilon}.
Let $\mathbf{u'} = (u'_n)_{n\geq1}$ be the sequence of derivatives of the $u_n$'s
and $\theta_k = \sup_{t\in[0,1]} |S_{k^2,h}(\mathbf{u'}(t))|$.
Since the sequence of derivatives $(u'_n)_{n\geq1}$ is computable, 
the same holds for the sequence $(|S_{k^2,h}(\mathbf{u'})|)_{k\geq1}$,
from which we get the computability 
of the sequence of reals $(\theta_k)_{k\geq1}$.
Computing rational approximations of the $\theta_k$'s up to $1$, we can define
a computable function $p \colon \N\setminus\{0\}\to\N$ such that $p(k)>\theta_k$.
By the mean value theorem, we get
\begin{align}\label{eq function p}
|S_{k^2,h}(\mathbf{u})(s) - S_{k^2,h}(\mathbf{u})(t)| 
\leq \theta_k |s-t| < p(k) |s-t|.
\end{align}
Define computable functions $a \colon \N\to\N$ 
and $q \colon \N\times \N \to \Q$ such that
\begin{align}
\label{eq a(k)}
&\text{for $k\geq1$}, &&2^{-a(k)}<\varepsilon/8p(k)&
\\
\label{eq q(k,i)}
&\text{for $i\leq2^{a(k)}$},&&
q(k,i)\in\Q \text{ approximates }
|S_{k^2,h}(\mathbf{u})(i 2^{-a(k)})|
\text{ up to } \varepsilon/8&
\end{align}
For each $k$, let 
\begin{align*}
X_k &= \{i \leq 2^{a(k)} \mid q(k,i)>(3/4)\varepsilon\}
\\
B_k& = [0,1] \cap \bigcup_{i\in X_k} ((i-1)2^{-a(k)}, (i+1) 2^{-a(k)}).
\end{align*}
If $t\in B_k$, say $t\in((i-1) 2^{-a(k)}, (i+1) 2^{-a(k)})$ with $i\in X_k$, 
then, applying successive inequalities
\eqref{eq function p},~\eqref{eq a(k)}, condition ~\eqref{eq q(k,i)}
and the definition of $X_k$, we get
\begin{align*}
|S_{k^2,h}(\mathbf{u})(t) | 
&\geq |S_{k^2,h}(\mathbf{u})(i 2^{-a(k)})| - p(k)2^{-a(k)}\\
&\geq \left(q(k,i) - \dfrac{\varepsilon}{8}\right) - \dfrac{\varepsilon}{8}\\
&> \dfrac{3\varepsilon}{4} - \dfrac{\varepsilon}{4} = \dfrac{\varepsilon}{2}.
\end{align*}
This proves inclusion $B_k \subseteq A_k^{\varepsilon/2}$.

If $t\in A_k^\varepsilon$, say $t\in[i 2^{-a(k)}, (i+1) 2^{-a(k)}]$,
then, again applying~\eqref{eq q(k,i)},~\eqref{eq function p}, 
recalling what means $t\in A_k^\varepsilon$, 
and applying inequality~\eqref{eq a(k)}, we get
\begin{eqnarray*}
q(k,i) &\geq & | S_{k^2,h}(\mathbf{u})(i 2^{-a(k)}) | - \varepsilon/8\\
        &\geq  &  (|S_{k^2,h}(\mathbf{u})(t)| - p(k)2^{-a(k)}) - \varepsilon/8 \\
          &>    & (\varepsilon - \varepsilon/8) - \varepsilon/8 = (3/4)\varepsilon
\end{eqnarray*}
hence, $i\in X_k$ and $t\in B_k$.
This proves inclusion $A_k^\varepsilon \subseteq B_k$.
Thus, we have $A_k^\varepsilon \subseteq B_k\subseteq A_k^{\varepsilon/2}$.
The first inclusion $A_k^\varepsilon \subseteq B_k$ ensures that the real $x$,
lying in infinitely many $A_k^\varepsilon$'s, also lies in infinitely many $B_k$'s.
The second inclusion and inequality~\eqref{eq mu Akvarepsilon} 
(applied with $\varepsilon/2$)
ensures that
\begin{align}\label{eq Schnorr mu Bk}
\lambda(B_k) \leq \lambda(A_k^{\varepsilon/2}) &<\alpha \dfrac{\ln(k)}{k^{2}}
\qquad\text{where $\alpha = \dfrac{8}{\varepsilon^2}  \left(1+\dfrac{17}{|h|K}\right)$}
\end{align}
hence, the series $\sum_{k\geq1} \lambda(B_k)$ converges.

To see that the sum $\sum_{k\geq1} \lambda(B_k)$ is computable, 
observe that, for every $L\geq2$,
\begin{enumerate}
\item[-]
the measure $\sum_{1\leq k\leq L}\lambda(B_k)$ is rational and can be uniformly computed from $L$ since $B_k$ is a finite union of rational intervals indexed by the finite set $X_k$ and the $X_k$'s are uniformly computable in $k$,
\item[-]
given any rational $\delta>0$ one can compute $L$ so that the tail $\sum_{k> L}\lambda(B_k)$ is smaller than $\delta$. Indeed,
applying inequalities~\eqref{eq Schnorr mu Bk} and the fact that $\ln(t) / t^2$ is decreasing for $t\geq2$,
we have, for $L\geq3$, 
\begin{align*}
\sum_{k>L} \lambda(B_k) 
&< \alpha \sum_{k>L} \dfrac{\ln(k)}{k^{2}}
< \alpha  \int_{L}^{+\infty} \dfrac{\ln(t)}{t^{2}} dt
= \alpha \frac{\ln(L)+1}{L}.
\end{align*}
\end{enumerate}
This shows that the family $(B_k)_{k\geq1}$ is a total Solovay test failed by~$x$.
\hfill$\qed$

\section{Proof of Theorem~\ref{thm:Sigma01}}\label{s:Sigma01}
%
\subsection{A needed lemma}

We  need the following result.

\begin{lemma}\label{l:measure frac U}
Let $A$ be a set of real numbers and 
let $\expa{A}=\Big\{ \{z\} :z\in A\Big\}$ be the set of fractional parts of elements of $A$. 

1. If $A$ is $\Sigma^0_1$ then so is $\expa{A}$
and $\measure{\expa{A}}\leq\measure{A}$.

2. If $\measure{A}$ is computable then so is $\measure{\expa{A}}$.
\end{lemma}

\begin{proof}
In case $A$ is a computable interval 
$(n+\alpha,p+\beta)$
where $n,p$ are integers, $n\leq p$, and $0\leq\alpha,\beta<1$ are computable,
the result is clear since
\begin{align*}
\expa{(n+\alpha,p+\beta)} =\left\{
\begin{array}{cl}
\protect{(\alpha,\beta)} & \text{ if }  n=p
\\
\protect{[0,\beta) \cup (\alpha,1)} & \text{ if } p=n+1
\\
\protect{[0,1)} & \text{ if }   p\geq n+2
\end{array}\right.
\end{align*}
Taking  a countable union we get item 1.

For item 2, observe that, up to a countable number of points,
a $\Sigma^0_1$ set $A$ is equal to a $\Sigma^0_1$ countable union 
$\bigcup_{k\in\N} I_{k}$ of pairwise disjoint intervals with rational endpoints.
In particular, $\measure{A}=\sum_{k\in\N}\measure{I_{k}}$. 
If $\measure{A}$ is computable then we can compute a function $n\mapsto p_{n}$
such that 
$\sum_{k> p_{n}}\measure{I_{k}} < 2^{-n}$.
Since $\measure{\expa{I_{k}}}\leq\measure{I_{k}}$, we see that
$\measure{\expa{\bigcup_{k>p_{n}} I_{k}}}$
 is a computable (in $n$) approximation of
$\measure{\expa{A}}$ up to~$2^{-n}$.
\end{proof}

\subsection{Proof of Theorem~\ref{thm:Sigma01} for Martin-L\"of randomness}
We prove the contrapositive of Theorem~\ref{thm:Sigma01}.
Assume $x\in [0,1]$ is not Martin-L\"of random and let 
$(u_n)_{n\geq 1}$ be a computable sequence of functions $[0,1]\to\R$
which is $(\ell_n)_{n\geq 1}$-Lipschitz,
where $(\ell_n)_{n\geq 1}$ is a computable sequence of rationals.
We construct a $\Sigma^0_1$ set $A$ which witnesses that the sequence
$(u_{n}(x))_{n\geq1}$ is not $\Sigma^0_1$-u.d.

Define
\begin{align*}
P(N,X)&=\frac{1}{N}\#\Big\{q: 1\leq q\leq N, \expa{x}\in X\Big\}.
\end{align*}
We first give the flavor of the proof.
Let $(V_n)_{n\geq 1}$ be a Martin-L\"of test failed by $x$,
that is,  $x\in\bigcap_{n\geq 1}V_n$.
Up to the extraction of a computable subsequence of $(V_n)_{n\geq 1}$,
one can suppose that $\measure{V_n}\leq 2^{-n-3}/\ell_n$, so that by 
$\ell_{n}$-Lipschitzness we have 
$\measure{u_n(V_n)}\leq 2^{-n-3}$.
Consider the set $A=\bigcup_{n\in\N} u_n(V_n)$. Then, 
\[
\measure{A}\leq\sum_{n\in\N}\measure{u_n(V_n)} 
\leq\sum_{n\in\N} 2^{-n-3} = 1/4,
\]
hence also 
\[
\measure{\expa{A}}\leq 1/4.
\]
Since $x$ is in all $V_{n}$'s, for every $n$ we have $\expa{u_n(x)}\in\expa{A}$
hence $P(N,\expa{A})=1$. This shows that the sequence $(u_{n}(x))_{n\geq1}$ is not 
u.d. relative to the set $\expa{A}$.
Unfortunately, the set $\expa{A}$ is not necessarily open. 

Now we slightly modify the $u_n(V_n)$'s to get a $\Sigma^0_1$ set.
Let $(I_{n,k})_{n,k\in\N}$ be a sequence of open intervals 
with rational endpoints such that 
$V_{n}=\bigcup_{k\in\N}I_{n,k}$ for all $n$
and the sequences 
of left and right endpoints
are computable.
Observing that $I_{n,k}\setminus(\bigcup_{p<k}I_{n,p})$ is
 a finite union of pairwise disjoint intervals (possibly not open), we write 
$V_{n}=\bigcup_{k\in\N}J_{n,k}$ where, for each $n$, the sequence
$(J_{n,k})_{k\in\N}$ consists of pairwise disjoint intervals (possibly not open)
such that the sequences
 of left and right endpoints
are computable.

Since the $u_{n}$'s are computable they are continuous and, by the mean value theorem, 
$u_{n}(J_{n,k})$ is an interval with computable endpoints
$\alpha_{k,n}$ and $\beta_{k,n}$ :
\begin{align*}
\alpha_{k,n}&=\min\{u_{n}(t) : t\in J_{n,k}\}
&& \beta_{k,n}=\max\{u_{n}(t) : t\in J_{n,k}\}
\end{align*}
Since $u_{n}$ is $\ell_{n}$-Lipschitz we have 
$\measure{u_{n}(J_{n,k})} \leq \ell \measure{J_{n,k}}$.
Consider the $\Sigma^0_1$ open sets 
\[
\omega_{n,k} = (\alpha_{n,k}-2^{-n-k-5}, \beta_{n,k}+2^{-n-k-5})
\]
and
\[\Omega_{n}=\bigcup_{k\in\N}\omega_{n,k} \ \ \ \text{ and  }\ \  \
\Omega=\bigcup_{n\in\N}\Omega_{n}.
\]
Since $u_{n}(J_{n,k})$ is an interval with endpoints $\alpha_{n,k}, \beta_{n,k}$,
the set $\Omega_{n}$ contains the $u_{n}(J_{n,k})$'s hence contains $u_{n}(V_{n})$. 
Also, 
\begin{align*}
\Omega_{n} &= u_{n}(V_{n}) \cup \bigcup_{k\in\N} 
(\alpha_{n,k}-2^{-n-k-5}, \alpha_{n,k}]  \cup [\beta_{n,k}, \beta_{n,k}+2^{-n-k-5})
& 
\\
\measure{\Omega_{n}} &\leq \measure{u_{n}(V_{n})} + 2\times\sum_{k\in\N}2^{-n-k- 5}
\leq 2^{-n-3}+2^{-n-3}=2^{-n-2}\\
 \measure{\Omega}&\leq 1/2.
\end{align*}
Thus,  $\Omega$ is a $\Sigma^0_1$ set and by Lemma~\ref{l:measure frac U},
 $\expa{\Omega}$ is also a $\Sigma^0_1$ set. 
 The  measure of $\expa{\Omega}$ is at most $1/2$
and  $P(N,\expa{\Omega})=1$ because 
 for every $n$ we have $x\in V_{n}$,
so 
$u_{n}(x)\in u_{n}(V_{n})\subseteq \Omega_{n}\subseteq \Omega$.
Thus, $(u_{n}(x))_{n\geq1}$ is not 
u.d. relative to the $\Sigma^0_1$ set~$\expa{\Omega}$.
\hfill{\qed}

\subsection{Proof of Theorem~\ref{thm:Sigma01} for Schnorr randomness}
In case $x$ is not Schnorr then in the above argument we have the extra condition that the sequence $(\measure{V_{n}})_{n\in\N}$ is computable.
We show how to computably approximate $\measure{\Omega}$ 
up to any given $\varepsilon>0$.
Let $p$ be such that $2^{-p}\leq\varepsilon$.

{\em Approximation $1$.}
The set $\Omega=\bigcup_{n\in\N}\Omega_{n}$ is approximated by 
$\bigcup_{n<p}\Omega_{n}$ up to measure $2^{-p-3}$ since
\begin{align*}
\measure{\bigcup_{n\geq p}\Omega_{n}}
&\leq\measure{\bigcup_{n\geq p}u_{n}(V_{n})}
+\sum_{n\geq p,\ k\in\N}2\times 2^{-n-k-5}
\\
&\leq\big(\sum_{n\geq p}2^{-n-3}\big)+ 2^{-p-3} 
\\
&< 2^{-p-1}.
\end{align*}

{\em Approximation $2$.}
Observe that 
for each $n$,
\\
- $V_{n}=\bigcup_{k\in\N}J_{n,k}$ is the disjoint union of the intervals
 $J_{n,k}$'s for $k\in\N$,
\\
 - the sequence of their rational endpoints is computable, and 
\\
- the measure of $V_{n}$
is computable uniformly in $n$.  
\\
Then, we can compute $q$ such that, for each $n<p$, the set 
$\bigcup_{k<q}J_{n,k}$ approximates~$V_{n}$ up to measure
$2^{-n-p-3}/\ell_{n}$.
We can also assume $q\geq p$.
By $\ell_{n}$-Lipschitzness, for each $n<p$, 
the set $\bigcup_{k<q}u_{n}(J_{n,k})$ approximates 
$u_{n}(V_{n})$ up to measure $2^{-n-p-3}$.
\linebreak
Finally, the set $\bigcup_{n<p,\ k<q}u_{n}(J_{n,k})$ approximates
$\bigcup_{n<p}u_{n}(J_{n,k})$ up to measure 
\linebreak
$\sum_{n<p}2^{-n-p-3} < 2^{-p-2}$.
\medskip

{\em Approximation $3$.}
For each $n$ and $k$ let 
\[
\rho_{n,k}
=(\alpha_{n,k}-2^{-n-k-5}, \alpha_{n,k}]  \cup [\beta_{n,k}, \beta_{n,k}+2^{-n-k-5}).
\]
The set $\bigcup_{k<q} \rho_{n,k}$ approximates $\bigcup_{k\in\N} \rho_{n,k}$
up to measure $2^{-n-q-3}$ because 
\begin{align*}
\measure{\bigcup_{k\geq q} \rho_{n,k}} 
&\leq \sum_{k\geq q} 2^{-n-k-4} = 2^{-n-q-3}
\end{align*}
Hence $\bigcup_{n< p,\ k<q} \rho_{n,k}$ approximates 
$\bigcup_{n<p,\ k\in\N} \rho_{n,k}$ up to measure 
$\sum_{n<p}2^{-n-q-3} < 2^{-q-2}$.
\medskip

{\em Conclusion.} 
The measure of the set 
\[
Z=\big(\bigcup_{n<p,\ k<q}u_{n}(J_{n,k})\big) \cup \big(\bigcup_{n< p,\ k<q} \rho_{n,k}\big)
\]
which is made of at most $3pq$ many intervals,
 is  computable.
Using 
Approximations~$2$ and~$3$, we see that the set $Z$ approximates 
\[
\bigcup_{n<p}\Omega_{n} 
= \big(\bigcup_{n<p,\ k\in\N}u_{n}(J_{n,k})\big) \cup \ \big(\bigcup_{n< p,\ k\in\N} \rho_{n,k}\big)
\]
 up to measure $2^{-p-2}+  2^{-q-2} \leq  2^{-p-1}$ (since $q\geq p$).
Using 
Approximation $1$, the set $\bigcup_{n<p}\Omega_{n}$ approximates $\Omega$
up to measure $2^{-p-1}$.
Thus, $Z$ approximates $\Omega$ up to measure $2^{-p}$ hence 
up to measure $\varepsilon$.
This concludes the proof of Theorem~\ref{thm:Sigma01}.
\hfill$\qed$
\bigskip
\bigskip

\noindent{\bf Acknowledgements:}
The authors are grateful to Ted Slaman for valuable discussions.
The authors are members of LIA SINFIN, Universit\'e de Paris-CNRS/Universidad de Buenos Aires-CONICET.
Becher is supported by grant STIC-Amsud 20-STIC-06 and PICT-2018-02315.

\bigskip

\bibliography{koksma}
\bigskip

\begin{minipage}[c]{\textwidth}
\small
\noindent
Ver\'onica Becher\\
 Departmento de  Computaci\'on,   Facultad de Ciencias Exactas y Naturales\\
 Universidad de Buenos Aires \& ICC CONICET\\
Pabell\'on I, Ciudad Universitaria, 1428 Buenos Aires, Argentina
\\  
{\tt vbecher@dc.uba.ar}
 \bigskip

\noindent
Serge Grigorieff\\
IRIF,  Universit\'e de Paris \& CNRS\\
Case 7014 -  75205 PARIS Cedex 13, France\\
{\tt seg@irif.fr}
\end{minipage}

\end{document}